\documentclass[reqno]{amsart}
\pdfoutput=1

\usepackage{enumerate}
\usepackage[colorlinks=true]{hyperref}
\usepackage{fullpage}
\usepackage{verbatim}
\usepackage{graphicx}
\usepackage{tikz}
\usepackage{amsmath}
\usepackage{amssymb}
\usepackage{array, booktabs}
\usepackage{amsfonts}
\usepackage{amsthm}
\usepackage{mathtools}
\usepackage{caption}
\usepackage{subcaption}
\usepackage{commath}
\usepackage{longtable}
\usepackage{graphicx}
\usepackage{epsfig}
\usepackage[all]{xy}

\usepackage{color}
\usepackage{skull}

\newtheorem{theorem}{Theorem}[section]
\newtheorem*{theorem*}{Theorem}
\newtheorem{lemma}[theorem]{Lemma}

\newtheorem{prop}[theorem]{Proposition}

\newtheorem*{conj*}{Conjecture}

\theoremstyle{remark} 
\newtheorem*{question*}{Question}

\newtheorem{remark}[theorem]{Remark}
\newtheorem{ex}[theorem]{Example}

\theoremstyle{definition} 
\newtheorem{defn}[theorem]{Definition}

\numberwithin{equation}{section}  

\newcommand{\OO}{\mathcal{O}}    
\newcommand{\ZZ}{\mathbb{Z}}     
\newcommand{\NN}{\mathbb{N}}     
\newcommand{\QQ}{\mathbb{Q}}      
\newcommand{\CC}{\mathbb{C}}      

\newcommand{\be}{\begin{equation}}
\newcommand{\ee}{\end{equation}}
\newcommand{\benn}{\begin{equation*}}
\newcommand{\eenn}{\end{equation*}}
\newcommand{\ba}{\begin{aligned}}
\newcommand{\ea}{\end{aligned}}
\newcommand{\bbm}{\begin{bmatrix}}
\newcommand{\ebm}{\end{bmatrix}}
\newcommand{\bpm}{\begin{pmatrix}}
\newcommand{\epm}{\end{pmatrix}}
\newcommand{\bi}{\begin{itemize}}
\newcommand{\ei}{\end{itemize}}

\newcommand{\an}[1]{\operatorname{an}}  

\newcommand{\Crit}{\mathrm{Crit}}

 \newcommand{\PGL}{\mathrm{PGL}}    
 
\newcolumntype{M}[1]{>{\centering\arraybackslash}m{#1}}

\newcommand{\Manoa}{M\=anoa}
\newcommand{\Hawaii}{Hawai\kern.05em`\kern.05em\relax i}

\title{Cubic postcritically finite polynomials defined over $\QQ$}

\author{Jacqueline Anderson}
\author{Michelle Manes}
\thanks{The second author's work partially supported by Simons collaboration grant \#359721.}
\author{Bella Tobin}

\address{Department of Mathematics, Bridgewater State University, Bridgewater, MA 02325}

\address{Department of Mathematics,
University of \Hawaii\ at \Manoa,
Honolulu, HI}

\address{Department of Mathematics, Oklahoma State University, Stillwater, OK 74078}

\email{jacqueline.anderson@bridgew.edu}
\email{mmanes@math.hawaii.edu}
\email{bella.tobin@okstate.edu}

\begin{document}

\maketitle

 \begin{abstract}
We  find all  post-critically finite (PCF) cubic polynomials defined over $\QQ$, up to conjugacy over $\PGL_2(\bar \QQ)$. We describe  normal forms that classify equivalence classes of cubic polynomials while respecting the field of definition. Applying known bounds on the coefficients of post-critically bounded polynomials to these normal forms simultaneously at all places of $\QQ$, we create a finite search space  of cubic polynomials over $\QQ$ that may be PCF. 
Using a computer search of these possibly PCF cubic polynomials, we find fifteen which are in fact PCF.
\end{abstract}

\section{Introduction}

Let $K$ be a number field, and let $f(z) \in K[z]$ have degree $d\geq 2$.   Consider iterates of $f$:
\[
f^n(z) := \underbrace{f\circ f\circ\cdots\circ f}_{n\text{ times}} (z),
 \quad \text{ and } \quad f^0(z) := z.
\]
The orbit of a point $\alpha \in \bar K$ is the set
$\OO_f(\alpha) = \left\{ f^n(\alpha) \mid n \geq 0 \right\}$.

Rather than studying individual polynomials, we consider equivalence classes of  polynomials under conjugation by affine elements $\phi \in\PGL_2(\bar K)$. For $\phi(z) = az + b \in \bar K[z]$, we define
\[
f^\phi = \phi \circ f \circ \phi^{-1}.
\] 
Note that $f$ and $ f^\phi$ have the same 
dynamical behavior over $\bar{K}$ in the sense that  $\phi$ maps the orbit $\OO_f(\alpha)$ to $\OO_{f^\phi}\left(\phi(\alpha)\right)$.  

Critical points of $f$ are the points $\alpha \in \bar K$ such that $f'(\alpha) = 0$. 
Branner and Hubbard write in \cite{MR945011} ``the main question to ask about a rational map is: \emph{what are the orbits under iteration of the critical points?}'' Of particular interest are functions for which all critical points have either a bounded or finite orbit.

\begin{defn}
A polynomial $f$ is \emph{postcritically finite}  (PCF) if the orbit of each critical point is finite.  A polynomial is \emph{postcritically bounded} with respect to a given absolute value if the orbit of each critical point is bounded with respect to that absolute value. 
\end{defn}

The study of PCF maps has a long history in complex dynamics, from Thurston's work in the early 1980s and continuing to the present day, for example~\cite{BK, DMEC, FG1, FPP, GV2, GV3}.  In~\cite{MdADS}, Silverman describes PCF maps as an analog of abelian varieties with complex multiplication, so these maps are  of particular interest in arithmetic dynamics as well.   For example, all quadratic post-critically finite rational maps over $\QQ$ have been found in~\cite{lukasmanesyap}, and many cubic post-critically finite polynomials over $\QQ$ have been found in~\cite{ingram2010}. 
 
\begin{theorem*}\label{all pcf}
There are exactly fifteen $\bar \QQ$ conjugacy classes of cubic PCF polynomials defined over~$\QQ$\textup{:}
\begin{align*}
(1) &\quad z^3   & (2) &\quad -z^3+1     &(3) &\quad -2z^3+3z^2+\frac{1}{2}    \\
(4) &\quad  -2z^3+3z^2  & (5) &\quad -z^3+\frac 32 z^2 -1 & (6) & \quad  2z^3-3z^2+1   \\
(7) &\quad 2z^3-3z^2+\frac{1}{2} & (8) & \quad  z^3-\frac{3}{2}z^2  &(9) &\quad   -3z^3+\frac{9}{2}z^2 \\
(10) &  \quad    -4z^3+6z^2-\frac 12    &(11)&  \quad  4z^3-6z^2+\frac 32   &(12)&\quad  3z^3-\frac 92 z^2+1 \\
(13)&\quad  -z^3+\frac 32 z^2 -1  &(14)&\quad  -\frac14 z^3+\frac32 z + 2    & (15)& \quad  -\frac{1}{28} z^3-\frac34 z + \frac72 
\end{align*}
\end{theorem*}

Of these, (1), (4), (6), (8), (10), and~(11) were found by Ingram in~\cite{ingram2010}. To complete the list, we  adapt Ingram's techniques as described below.

 Let $K$ be a number field, and let $f(z) \in K[z]$ be a cubic polynomial. Critical points of $f$ are roots of the quadratic polynomial $f'(z) \in K[z]$, so there are three possibilities:
\begin{enumerate}[(a)]
\item\label{case:birat} There are two distinct critical points: $\gamma_1\neq \gamma_2$, and they are both $K$-rational.
\item\label{case:biirat} There are two distinct critical points $\gamma_1 \neq \gamma_2 $ with $K(\gamma_1) = K(\gamma_2)$  a quadratic extension of $K$.
\item\label{case:uni} There is exactly one critical point, $\gamma \in K$.
\end{enumerate} 
 In the first two cases, we say that $f$ is \emph{bicritical}. In the third case, we say  $f$ is \emph{unicritical}.
In determining a complete list of cubic PCF polynomials defined over $\QQ[z]$, we treat each of these cases separately.

\begin{enumerate}
\item
For each of cases~\eqref{case:birat}--\eqref{case:uni} above, find a normal form for cubic polynomials such that every cubic polynomial over $\QQ[z]$ is conjugate to a map in one of these forms, and the conjugation respects the field of definition for the given case.
\item
For a map to be PCF, it must be post-critically bounded in each absolute value. 
Find archimedean and $p$-adic bounds on the coefficients for maps in the normal forms to be post-critically bounded. 
\item
Use the bounds in (2) to create a finite search space of possibly PCF maps.
\item
For each map in the finite search space, test if it is PCF or not.
\end{enumerate}

\subsection{Outline}
We begin in Section~\ref{sec:unicrit} by treating  the special case of a polynomial with a unique critical point. In Section~\ref{sec:normal}, we find the normal forms needed in Step (1) of the algorithm above. Section~\ref{sec:CoeffBounds} provides the  coefficient bounds described in Step (2). 
Finally, Section~\ref{cubicpcfsection} describes the algorithms and provides the complete list of PCF cubic polynomials defined over $\QQ$.

\subsection*{Acknowledgements}
Material from this article forms a part of the third author's Ph.D.~thesis. The authors thank the committee members for helpful comments: Rosie Alegado, Pavel Guerzhoy, Piper H, Ruth Haas, and Rob Harron. We thank Sarah Koch for helpful comments and conversation. 

The project was completed during a SQuaRE at the American Institute for Mathematics. The authors thank AIM for providing a supportive environment.

This material is based upon work supported by and while the second author served at the National Science Foundation. Any
opinion, findings, and conclusions or recommendations expressed in this material are those of the authors
and do not necessarily reflect the views of the National Science Foundation.

\section{Unicritical PCF polynomials}\label{sec:unicrit}

We begin by considering  unicritical PCF polynomials.
First, we will determine a normal form for  unicritical polynomials of arbitrary degree defined over a number field $K$. 
 In \cite{MR3890968}, Buff studied unicritical polynomials from a complex dynamics point of view, and he used that work to answer questions of Milnor and of Baker and DeMarco. Some of his preliminary work overlaps with the work here, specifically the normal form in Theorem~\ref{unithm} and the bound on $|a|$ in Proposition~\ref{archimedeanuni}. Because Buff was working over $\CC$, he did not consider questions about field of definition. Therefore, we provide full proofs of these results from a more arithmetic point of view.

\begin{theorem}
\label{unithm}
Let $f(z) \in K[z]$ be a degree $d$ unicritical polynomial. Then either $f(z)$ is $\bar{K}$-conjugate to $z^d$, or $f$ is conjugate to a unique polynomial of the form $$az^d+1 \in K[z].$$ 
\end{theorem}
\begin{proof}
Without loss of generality, we may replace $f$ by a conjugate map where the unique critical point $\gamma$ is moved to $0$. Since $\gamma \in K$, this does not change the field of definition. 
So we assume that $f(z) = bz^d+c \in K[z]$. 

If $c = 0$, then $f(z) = bz^d$ for $b \in K^\times$. Letting $\phi (z) = b^\frac{1}{d-1}z$, we have $f^\phi (z) = z^d$. 

 Now, assume $c \neq 0$. Conjugating by $\phi(z) = \frac{z}{c} $ gives 
$$f^\phi (z) = bc^{d-1}z^d + 1.$$ 
Since $b,c \in K^\times,$ then $bc^{d-1} \in K^\times$. Letting $a = bc^{d-1}$ gives the result.

Finally, $\phi$ is the only affine map in $\PGL_2(\bar K)$ fixing $0$  and satisfying $f^\phi(0) = 1$. Therefore, $f(z)$ is $\bar{K}$-conjugate to $az^d+1 \in K[z]$ for a unique $a \in K^\times$. 
\end{proof}

Theorem~\ref{unithm} implies that up to conjugacy every  unicritical polynomial $f\in K[z]$ is a power map or of the form $az^d+1$. In both cases $\Crit (f) = \lbrace 0 \rbrace$. If $f$ is a power map then $f(0) = 0$, hence $f$ is PCF. Therefore, in order to completely describe all other PCF  unicritical polynomials in $\QQ[z]$ (of any degree), we need only consider those of the form $f(z) = az^d+1$ for $ a \in \QQ^\times$. 

\begin{prop}[\cite{MR3890968}, Corollary 8]
\label{archimedeanuni}
If $f(z) = az^d+1 \in K[z]$ is post-critically finite, then $\abs{a}\leq 2$. 
\end{prop}
\begin{proof}
Suppose $\abs{a}> 2$ and $\abs{\alpha} \geq 1$. Then 
$$\abs{f(\alpha)} = \abs{a\alpha^d + 1} > |\alpha|.$$
Inductively,  $\alpha$ must be a wandering point. Since $\Crit (f) = \lbrace 0\rbrace$ and $f(0) = 1$, we see that  $f$ is not PCF. Therefore, if $f \in K[z]$ is PCF it must be that $\abs{a}\leq 2$. 
\end{proof}

\begin{theorem}
\label{uniprop}
Let $f(z) = az^d+1 \in \QQ[z]$ and $d\geq 2$. For $d$ even,  $f$ is PCF if and only if $a \in \lbrace -2, -1\rbrace$. For $d$ odd, $f$ is PCF if and only if $a = -1$.  
\end{theorem}

\begin{proof}
Suppose $|a|_p>1$ for some prime $p$. If $|z|_p \geq 1$, then $|f(z)|_p = |az^d+1|_p  = |az^d|_p> |z|_p$, so $\alpha$ is a wandering point if there exists $n \geq 0$ such that $|f^n(\alpha)|_p \geq 1$. In particular, $f(0)=1$, so the critical point $0$ is a wandering point and $f$ is not PCF. We conclude that for all primes $p$,  $\abs{a}_p \leq 1$; hence $a \in \ZZ$. 
By Proposition \ref{archimedeanuni}, $\abs{a}\leq 2$, so $a\in \lbrace \pm 1, \pm 2\rbrace$. 

Suppose that $\abs{\alpha}> 2$. Then 
\[
    \abs{f(\alpha)} = \abs{a\alpha^d +1} > 2^{d-1}\abs{\alpha} - 1 > |\alpha|.
    \]
Inductively, $\alpha$ must be a wandering point for $f$.

If $a = 1$, then
$f^{3} (0) = 2^d+1$, so $0$ must be a wandering point. 
 If $a = 2$, then $f^{2}(0) =3$, so $0$ must be a wandering point. 
 If $a = -1$, then $f^2(0) = 0$, so $f$ is PCF. 

Finally, consider the case $a = -2$. If $d$ is even then $f^{2} (0) = f^3(0) = -1$, so $f$ is PCF. If $d$ is odd, then $f^3(0) = 3$, so $0$ is a wandering point. 
 \end{proof}

\section{Normal forms for  bicritical polynomials}\label{sec:normal}
Cubic polynomials have been studied extensively in complex dynamics, e.g. \cite{MR2600536, MR2670510, MR945011, MR1194004, MR2508263}, and in arithmetic dynamics, e.g. \cite{ingram2010}. All of these use the Branner-Hubbard normal form, sometimes also called the monic centered form:
$$
F(z) = z^3+Az+B
\quad \text{with critical points } \pm \alpha \text{ where } \alpha = \sqrt{\frac{-A}{3}}.
$$
This form may be preferred in complex dynamics, but it is not ideal in arithmetic dynamics because it does not preserve the field of definition of the polynomial. 
For example, in~\cite{ingram2010}, Ingram shows that if $K$ is a number field and $F(z) \in K[z]$ is PCF, then the pairs $(A,B)$ are in a finite  computable set, and he finds the set in the case $F(z) \in \QQ(z)$. However,  our Table~\ref{CubicTable} shows that fewer than half of the PCF cubic polynomials defined over $\QQ$ are conjugate to some $F(z) \in \QQ[z]$ in the Branner-Hubbard form.

\begin{ex}
Consider the PCF polynomial $f \in \QQ [z]$ given by 
$f(z) = 3z^3-\frac 92 z^2+1$. Conjugating by 
\[
\phi (z) = \sqrt{3}z- \frac{\sqrt{3}}{2} \text{ gives }
f^\phi (z) =  z^3-\frac{9}{4}z-\frac{\sqrt{3}}{4} \not\in \QQ[z].
\]
\end{ex}

In this section, we describe  normal forms for cubic bicritical polynomials, one for the case of rational critical points and one for the case of irrational critical points.
These cases are not disjoint, but both are necessary  to exhaustively list all PCF cubic polynomials. It is a simple matter to check that our final list of cubic polynomials contains no conjugate maps, so this is of no concern.

\begin{ex}\label{ex:hasaut}
Let 
\[
f_1(z) = \frac{z^3}{4}-\frac{3 z}{2}, \text{ so } \Crit(f_1) = \left\{\pm \sqrt{2} \right\}.
\]
Moving the two critical points to 0 and 1 gives the polynomial $g_1(z) = 2 z^3 - 3 z^2 + 1$.
These conjugate polynomials --- one with rational critical points and one with irrational critical points --- are both defined over $\QQ$.
\end{ex}

If $f(z) \in K[z]$ has two rational critical points, we may conjugate to move them to 0 and 1 without changing the field of definition. From~\cite[Proposition~2.3]{manesbelyi}, we know that there is a \emph{unique} conjugacy class of bicritical polynomials of degree $d\geq 3$ with fixed critical points  $\gamma_1$ and $\gamma_2$, and with prescribed ramification at the two critical points.
Moreover, we have a formula for this polynomial when  $\{\gamma_1, \gamma_2\} = \{0,1\}$.
Call the polynomial $\mathcal{B}_{d,k}(z)$. Since the critical points are at $0$ and $1$ and the polynomial has degree $d$, we have
\[
\mathcal{B}'_{d,k}(z) = c z^{d-k-1}(z-1)^k 
\]
for some $1\leq k <d -1$ and some constant $c$. So $d-k$ is the ramification index of $\mathcal{B}_{d,k}(z)$ at the critical point $0$, and $k+1$ is the ramification index at the critical point $1$. Expanding with the binomial theorem, integrating term-by-term, and requiring that the two critical points are fixed gives:
\begin{equation}\label{eqn:BelyiFrom}
\mathcal{B}_{d,k}(z)= \left(\frac{1}{k!} \prod_{j=0}^k(d-j)\right)z^{d-k}\sum\limits_{i=0}^k \frac{(-1)^{i}}{(d-k+i)}{k \choose i} z^i. 
\end{equation}


Since we are concerned with the case $d=3$, necessarily $k=1$, giving the polynomial:
\begin{equation}\label{eqn:CubicBelyiFrom}
\mathcal{B}_{3,1}(z)= -2z^3 + 3z^2.
\end{equation}

\begin{prop}
\label{familyprop}
Let $g \in K[z]$ be a bicritical polynomial of degree $d\geq 3$ with $\Crit(g) = \{\gamma_1, \gamma_2\} \subseteq K$. There exists  an element $\phi \in \PGL_2 (K)$ such that $g^\phi = a\mathcal{B}_{d,k} +c$ for some $k\in \NN$ and some $a, c \in K$.
\end{prop}

\begin{proof}
Let $g \in K[z]$ with critical points $\gamma_1, \gamma_2\in K$. 
Choose $k \in \NN$  such that $d-k$ is the ramification index of $\gamma_1$ and $k+1$ is the ramification index of $\gamma_2$. Define $\phi (z) = \frac{z-\gamma_1}{\gamma_2-\gamma_1}\in \PGL_2(K)$, which  moves the critical points to $0$ and $1$, respectively.

If $f(z) = g^\phi (z)$, then $f$ has critical points at $0$ and $1$ and degree $d$, so
\[
f'(z) = \alpha z^{d-k-1}(z-1)^{k} = a \mathcal{B}'_{d,k}(z)
\]
for some $a \in \bar{K}^\times$.  

Then $f(z) = a \mathcal{B}_{d,k}(z) + c$, and since $f = g^\phi$ where both $g, \phi \in K[z]$, we have $a, c \in K$.
\end{proof}

We now consider a normal form for cubic polynomials $g \in K[z]$ with critical points in a quadratic extension of $K$.

Let $D \in \mathcal{O}_K^\times$ and let $d\geq 3$ be odd.  We define a polynomial $\mathcal{P}_{d,D}(z) \in K[z]$ by the following conditions:
\begin{itemize}
\item
$\mathcal{P}'_{d,D}(z) = (z^2 - D)^{(d-1)/2}$.
\item
$\mathcal{P}_{d,D}(0) = 0$.
\end{itemize}

Then $\mathcal{P}_{d,D}(z)$  is a bicritical polynomial  having critical points $\left\{\pm \sqrt D\right\}$ each with ramification index $ \frac{d+1}2$. Just as with the polynomials $\mathcal{B}_{d,k}(z)$, we expand the derivative $ \mathcal{P}'_{d,D}(z)$ using the binomial theorem, integrate term-by-term, and use the fact that~$0$ is fixed to find a formula for these polynomials:
\begin{equation}\label{eq:irratform}
\mathcal{P}_{d,D}(z) = \sum_{j=0}^{\frac{d-1}2} (-D)^{\frac{d-1}2-j}\binom{\frac{d-1}2}{j}\frac{z^{2j+1}}{2j+1}.
\end{equation}
Of particular interest in the sequel is the cubic case:
\begin{equation}\label{eq:deg3irrat}
\mathcal{P}_{3,D}(z) = \frac{z^3}{3} - Dz.
\end{equation}

\begin{prop}\label{prop:MoveCritPts}
Let $g(z) \in K[z]$ be a bicritical polynomial of degree $d\geq 3$. Suppose that $\Crit(g) = \{\gamma_1,\gamma_2\} \not\subset K$. Then $g$ is conjugate to a map of the form $a \mathcal{P}_{d,D}(z) + c$ for some $a, c \in K$ and some $D \in \mathcal{O}_K^\times / \mathcal{O}_K^2$.
\end{prop}

\begin{proof}
By definition, $ \{\gamma_1,\gamma_2\} $ are roots of the  polynomial $g'(z) \in K[z]$. Since they are not in $K$, we must have  $g'(z) = \alpha (h(z))^\beta$ where $h \in K[z]$ is an irreducible quadratic polynomial. 
Note: In this case, $d$ is odd and the ramification index of each critical point is  $ \frac{d+1}2$.

Therefore there are $m, n \in K$ with $n\neq 0$ and   $D \in \mathcal{O}_K^\times / \mathcal{O}_K^2$ such that $\gamma_1 = m + n\sqrt{D}$ and  $\gamma_2 = m -n\sqrt{D}$.
Consider 
\[
\phi(z) = \frac{z-m}{n} \in K[z],
\quad \text{ which satisfies }
\phi(\gamma_1) = \sqrt{D} \text{ and } \phi(\gamma_2) = -\sqrt{D}.
\]
 Define $f(z) = g^\phi(z)$. Since $g, \phi \in K[z]$, we have $f(z) \in K[z]$. Hence $f'(z) \in K[z]$.
Furthermore, $\Crit(f) = \Crit(g^\phi) = \phi(\Crit(g)) = \{\pm \sqrt{D} \}$. Therefore, $f'(z) = a(z^2 - D)^{(d-1)/2}$ for some $a \in K$, which means that $f(z) =  a  \mathcal{P}_{d,D}(z) + c \in K[z]$.
\end{proof}

\section{Coefficient Bounds for PCF cubic polynomials over $\QQ$}\label{sec:CoeffBounds}
From Corollary 1.2 in \cite{ingram2010},  for any number field $K$ there are finitely many conjugacy classes of post-critically finite polynomial maps of degree $d$ in $K[z]$. We would like to use the  normal forms in Section~\ref{sec:normal} to determine a representative of each conjugacy class of PCF cubic polynomials over $\QQ$. Many of these results can be extended to  bicritical maps of arbitrary degree (see~\cite{tobin}).

Let $f(z) = a_dz^d+a_{d-1}z^{d-1}+\ldots +a_1z+a_0 \in K[z]$.
Following Ingram \cite{ingram2010}, we set the following notation:
\begin{align*}
(2d)_\nu &= \begin{cases} 1 & \nu \text{ is non-archimedean }\\
2d & \nu \text{ is archimedean}
\end{cases}\\
C_{f,\nu} &= (2d)_\nu \max\limits_{0\leq i<d} \left\lbrace 1, \abs{\frac{a_i}{a_d}}_\nu^{\frac{1}{d-i}}, \abs{a_d}_\nu^{-\frac{1}{d-1}}\right\rbrace.
\end{align*}



The following lemma show that $C_{f,\nu}$ gives an effective $\nu$-adic bound for preperiodic points (points with finite orbit) of a polynomial $f(z) \in \QQ[z]$. Applying this bound to the critical points will, in turn, give $\nu$-adic bounds on the coefficients for PCF polynomials. Ingram  uses $C_{f,\nu}$ in exaclty this way in~\cite{ingram2010} without stating and proving a lemma of this sort. We provide Lemma~\ref{cor1} and its proof for clarity and completeness.

\begin{lemma}
\label{cor1}
Let $f(z) \in \QQ[z]$ be a polynomial of degree $d\geq 2$. For $\alpha\in \QQ$, if there exists $\nu \in M_\QQ$ and $n\in\NN$ such that $$\abs{f^n(\alpha)}_\nu >C_{f,\nu},$$ then $\alpha$ must be a wandering point (have infinite orbit) for $f$. 
\end{lemma}

\begin{proof}
First, notice that $\alpha$ is a wandering point if and only if $f^n(\alpha)$ is a wandering point for all $n \in \NN$, so without loss of generality, assume  $\abs{\alpha}_\nu>C_{f,\nu}$ for some $\nu \in M_K$. We will show that $\alpha$ is a wandering point by proving that whenever $\abs{\alpha}_\nu>C_{f,\nu}$, we must have $\abs{f(\alpha)}_\nu > \abs{\alpha}_\nu$.

If $\nu$ is non-archimedean, then $\abs{\alpha}_\nu > \abs{\frac{a_i}{a_d}}_\nu^{\frac{1}{d-i}}$ guarantees that 
$\abs{a_d \alpha^d}_\nu >  \abs{a_i \alpha^i}_\nu$
 for all $i<d$, so we have
\[ \abs{f(\alpha)}_\nu = \abs{ \sum_{i=0}^d a_i \alpha^i}_\nu = \abs{a_d \alpha^d}_\nu > \abs{\alpha}_\nu.\]
The inequality above comes from the fact that $\abs{\alpha}_\nu > C_{f,\nu} \geq  \abs{a_d}_\nu^{-\frac{1}{d-1}}$.

If $\nu$ is archimedean, then starting with $\abs{\alpha}_\nu > 2d \abs{\frac{a_i}{a_d}}_\nu^{\frac{1}{d-i}}$, we see that
\[ \abs{a_d \alpha^d}_\nu > \max_{0 \leq i <d}\left\{ (2d)^{d-i} \abs{a_i \alpha^i}_\nu \right\} \geq 2d \max_{0 \leq i <d} \left\{\abs{a_i \alpha^i}_\nu\right\},\]
and so we have
\[ \abs{f(\alpha)}_\nu  =  \abs{ \sum_{i=0}^d a_i \alpha^i}_\nu  \geq \abs{a_d\alpha^d}_\nu - d \max_{0 \leq i <d} \abs{a_i \alpha^i}_\nu > \frac12 \abs{a_d \alpha^d}_\nu.\]
Finally, it follows from $\abs{\alpha}_\nu > 2d \abs{a_d}_\nu^{-\frac{1}{d-1}}$ that
\[ \frac12 \abs{a_d \alpha^d}_\nu > \frac12 (2d)^{d-1} \abs{\alpha}_\nu > \abs{\alpha}_\nu,\]
as desired.
\end{proof}


\subsection{PCF cubics with rational critical points}\label{sec:pcf-belyi-bounds}
We begin by specializing Lemma~\ref{cor1} to bicritical cubic polynomials with rational critical points, using the normal form in Proposition~\ref{familyprop}. 

\begin{lemma}
\label{wanderingpointcor}
Let 
\[
f(z) = a\mathcal{B}_{3,1}+c = a(-2z^3+3z^2) + c \in \QQ[z]
\]
 be a bicritical polynomial and let $\alpha \in \QQ$. If there exist $\nu \in M_\QQ$ and $n\in \NN$ such that 
 $$
 \abs{f^n(\alpha)}_\nu > C_{f,\nu}=
 (6)_\nu \max \left\lbrace 1, \abs{\frac{3}{2}}_\nu, \abs{\frac{1}{2a}}_\nu^{\frac{1}{2}}, \abs{\frac{c}{2a}}_\nu^{\frac{1}{3}}\right\rbrace,
 $$
then $\alpha$ is a wandering point for $f$.  
\end{lemma}

\begin{proof}
The result follows immediately from applying the definition of  $C_{f,\nu}$ and Lemma~\ref{cor1} to the coefficients of $f(z)$.
\end{proof}

\begin{remark}
\label{remarkaboutabs}
 Let 
 $f(z) =  a(-2z^3+3z^2) + c \in \QQ[z]$, so $\Crit(f) = \{0,1\}$.  
 If $f$ is PCF then  every element in the orbits of $0$ and $1$ must be bounded by $C_{f,\nu}$. 
 In particular, 
 \[
 \abs{f(1)}_{nu} = \abs{a+c}_\nu\leq C_{f,\nu} \text{ and } \abs{f(0)}_\nu = \abs{c}_\nu\leq C_{f,\nu}.
 \]
  Thus if $f$ is PCF, then $\max\lbrace \abs{c}_\nu, \abs{a+c}_\nu \rbrace \leq C_{f,\nu}$ for all $\nu \in M_\QQ$. 
For every non-archimedean place $\nu$, this means
 $\max \lbrace\abs{a}_\nu, \abs{c}_\nu \rbrace \leq C_{f,\nu}$.
\end{remark}

Using the bound given above, we can find bounds on the absolute values of the parameters $a$ and $c$ of a PCF polynomial of the form $f(z) =  a(-2z^3+3z^2) + c  \in \QQ[z]$.
We begin with an archimedean bound on the parameter $a$.

\begin{lemma}
\label{lem:a-arch}
Let $f(z) = a (-2z^3+3z^2) + c \in \QQ[z]$. If $f$ is PCF, then $\abs{a}<4$. 
\end{lemma}
\begin{proof}
Suppose $\abs{a}\geq 4$ and $\abs{\alpha}\geq \max\{\abs{c}, 2\}$. Then
$$\abs{f(\alpha)} = \abs{a\alpha^{d-1}\left(-(d-1)\alpha+d\right)+c},$$ and a straightforward calculation shows that 
$\abs{f(\alpha)} >\abs{\alpha}$.
%
If $\abs{c}\geq 2$, then $0$ must be a wandering point. If $\abs{c}<2$, then 
$$\abs{a+c}\geq \abs{a}-\abs{c}>2,$$ 
so $1$ must be a wandering point. Thus, if $f$ is PCF, we must have $\abs{a}<4$.
\end{proof}

The following lemmas  give  $p$-adic bounds on the parameters $a$ and $c$ when $f$ is PCF.

\begin{lemma}
\label{cfvnonarch}
If $f(z) =  a(-2z^3+3z^2) + c \in \QQ[z]$ is PCF then for non-archimedean $\nu \in M_\QQ$  
$$
C_{f,\nu} = \max  \left \lbrace 1, \abs{\frac 3 2}_\nu, \abs{\frac{1}{2a}}_\nu^{\frac{1}{2}} \right \rbrace.
$$ 
\end{lemma}

\begin{proof}
Let $f(z) = a(-2z^3+3z^2) + c \in \QQ[z] $ and $\nu\in M_\QQ$ be non-archimedean.
From Lemma~\ref{wanderingpointcor},
\[
C_{f,\nu} = \max \left\lbrace 1, \abs{\frac{3}{2}}_\nu, \abs{\frac{1}{2a}}_\nu^{\frac{1}{2}}, \abs{\frac{c}{2a}}_\nu^{\frac{1}{3}}\right\rbrace.
\]
Suppose 
\[
C_{f,\nu} =
\abs{\frac{c}{2a}}_\nu^{\frac{1}{3}} > \abs{\frac{1}{2a}}_\nu^{\frac{1}{2}}; \quad \text{ then }
 \abs{c}_\nu^{2} > \abs{\frac{1}{2a}}_\nu.
\]

However, since $f$ is PCF,   
\[
\abs{c}_\nu \leq C_{f,\nu}= \abs{\frac{c}{2a}}_\nu^{\frac{1}{3}}, \quad \text{ so } \abs{c}_\nu^{2} \leq \abs{\frac{1}{2a}}_\nu,
\]
giving  a contradiction. 
\end{proof}
Notice that the statement above holds for $a,c \in K$ and $\nu \in M_K$ for any number field $K$ and the proof is identical. 

\begin{lemma}
\label{prodlemma}
Let $f(z) = a(-2z^3+3z^2) + c \in \QQ[z]$ be PCF, let $p$ be an odd prime, and let $\abs{\cdot}_p$ be the $p$-adic absolute value. 
Then $|a|_p\leq 1$ and $\abs{c}_p^{2}\leq \abs{a}_p^{-1}$.
\end{lemma} 

\begin{proof}
From Lemma \ref{cfvnonarch},
$$
C_{f,p}=
\max  \left \lbrace 1, \abs{\frac 3 2}_p, \abs{\frac{1}{2a}}_p^{\frac{1}{2}} \right \rbrace
 =\max  \left \lbrace 1, \abs{3}_p, \abs{a}_p^{-\frac{1}{2}} \right \rbrace 
 =\max  \left \lbrace 1, \abs{a}_p^{-\frac{1}{2}} \right \rbrace.
 $$

There are two distinct cases:
\begin{enumerate}
    \item[(1)] $C_{f,p} = 1$, or
    \item[(2)] $C_{f,p} = \abs{a}_p^{-\frac{1}{2}} >1.$
\end{enumerate}

First, suppose $C_{f, p} =1 \geq \abs{a}_p^{-\frac{1}{2}}$. Then 
 $\abs{a}_p\geq 1$. 
 However, since $f$ is PCF, $$\abs{a}_p, \abs{c}_p \leq C_{f,p}=1.$$ Therefore
$\abs{a}_p = 1$, $\abs{a}_p^{-1}= 1$, and $\abs{c}_p^{2}\leq 1=\abs{a}_p^{-1} $. 

Now, suppose $C_{f,p} = \abs{a}_p^{-\frac{1}{2}}>1$.  Then $\abs{a}_p<1$, as desired. Furthermore, since $f$ is PCF, 
\[
\abs{c}_p \leq C_{f,p}=\abs{a}_p^{-\frac{1}{2}}.\qedhere
\]
\end{proof}

\begin{lemma}
\label{kis1}
Let $f(z) = a(-2z^3+3z^2) + c \in \QQ[z]$ be PCF. Then 
\[
\abs{2a}_2\leq 1 \quad \text{ and } \quad \abs{2c}_2\leq 1.
\]
In fact, $2a \in \ZZ$.
\end{lemma}

\begin{proof}
From  Lemma~\ref{prodlemma}, we have  $\abs{2a}_p\leq 1$ for all odd primes $p$, so $2a \in \ZZ$ will follow immediately once we know that $|2a|_2 \leq 1$.

From Lemma \ref{cfvnonarch}, 
\begin{equation}
C_{f,2} =
 \max  \left \lbrace 1, \abs{\frac 3 2}_2, \abs{\frac{1}{2a}}_2^{\frac{1}{2}} \right \rbrace
 = \max  \left \lbrace  2, \abs{\frac{1}{2a}}_2^{\frac{1}{2}}  \right \rbrace.
\end{equation}

Suppose $C_{f,2} =2$: 
Since $f$ is PCF, both $\abs{a}_2$ and $\abs{c}_2 \leq 2.$ Therefore, both $\abs{2a}_2$ and $\abs{2c}_2\leq 1$ as desired.

Suppose $C_{f,2} =\abs{\frac{1}{2a}}_2^{\frac{1}{2}} > 2 $: Then
\begin{equation}
     \abs{2a}_2 < \frac{1}{4} < 1. \label{a(d-1)lessthan1}
\end{equation}

By Lemma~\ref{lem:a-arch},  $\abs{a}<4$, so since $2a \in \ZZ$, we must have
\begin{equation}
\label{alessthan4}
 a \in \left\lbrace \frac{n}{2}: 1\leq \abs{n} <8 \right\rbrace. 
\end{equation}
 However, all of these possible $a$-values fail to satisfy equation~\eqref{a(d-1)lessthan1}, so the case $C_{f,2} =\abs{\frac{1}{2a}}_2^{\frac{1}{2}} > 2 $ does not happen.
Therefore, if $f$ is PCF then $C_{f,2} =2$, and both
$\abs{2a}_2 $ and $\abs{2c}_2\leq 1$ as desired. 
\end{proof}

\begin{prop}
\label{prop:Belyiformalist}
If $f$ is a cubic PCF polynomial of the form $a\mathcal{B}_{d,k}(z)+c \in \QQ[z]$, then 
$$
\pm a\in \left\lbrace \frac12,  1,  \frac 32,  2,  \frac52,  3,  \frac{7}{2}\right\rbrace
\text{ and } 
\pm c \in \left\lbrace 0 ,  1,  \frac 12,  \frac 32,  2 \right\rbrace.$$
\end{prop}
\begin{proof}
The result for $a$ follows from equation~\eqref{alessthan4} and  Lemma \ref{kis1}. 

Given the list for $a$, we see that $\abs{a}_p = 1$ for any prime $p \not\in\lbrace 2,3,5,7\rbrace$.  For $p \in \lbrace3,5,7\rbrace$, we have $\abs{a}_p\geq \frac{1}{p}$, so $\abs{a}_p^{-1}\leq p$.  Using Lemma~\ref{prodlemma}, we conclude that  $\abs{c}_p \leq 1$ in both cases.
 Combining this with the fact that $|2c|_2  \leq 1$ from Lemma~\ref{kis1}, we see that $\abs{2c}_p\leq 1$ for all primes $p$. That is,  $2c\in\ZZ$.

We will show that $|c|< \frac 52$. Suppose that $a$ is contained in the above list and $|\alpha|\geq |c|\geq\frac{5}{2}$. Then
 $$
 \abs{f(\alpha)}
 \geq \abs{a}\abs{\alpha}^2 \abs{-2\alpha+3}-\abs{c},$$ 
 and this implies $\abs{f(\alpha)} > \abs{\alpha}$. 
Hence $\alpha$ is a wandering point for $f$. Then $ c = f(0)$ must be a wandering point for $f$, in which case $f$ would not be PCF. The result for $c$ follows. 
\end{proof}

\subsection{PCF cubics with irrational critical points}\label{sec:pcf-non-belyi-bounds}
As in Section~\ref{sec:pcf-belyi-bounds}, we can use the bound $C_{f,\nu}$ to find bounds on the (archimedean and non-archimedean) absolute values of the parameters $a,c$ and $D$ of a PCF polynomial of the form $f(z) = a\mathcal{P}_{3,D}+c \in \QQ[z]$. Unlike in Section ~\ref{sec:pcf-belyi-bounds}, the bounds are not given explicitly. Instead, we will determine restrictions on the relationships between the three parameters. In Theorem~\ref{cubicbicriticalthm-irrat}, we  use these relationships to implement an algorithm that determines a finite set of triples $(D,a,c)$ for which the polynomial $f(z) = a\mathcal{P}_{3,D}+c \in \QQ[z]$ is possibly PCF. 

\begin{prop} \label{prop:aD_vals}
Let $f(z) = a(z^3/3 - Dz) + c \in \QQ[z]$. If $f$ is PCF, then 
\[
\pm aD \in \left\{ \frac 3 4, \frac 3 2, \frac 9 4, 3, \frac{15}{4}, \frac 9 2, \frac{21}{4} \right\}.
\]
\end{prop}

\begin{proof}
Let $\phi(z) = \frac{z-\sqrt{D}}{-2\sqrt{D}}$. Then $$f^\phi (z) = \frac{-2}{3}aD (-2z^3+3z^2)+\frac{aD}{3}-\frac{c-\sqrt{D}}{2\sqrt{D}}.$$ None of the bounds on $a$ in the previous section depended on the fact that $c\in \QQ$, so we may apply them to $f^\phi$. From Proposition \ref{prop:Belyiformalist}, we have that
$$\pm\frac{2}{3}aD \in \left\lbrace \frac12 , 1, \frac32, 2, \frac52, 3, \frac72 \right\rbrace. \qedhere$$ 
\end{proof}

\begin{lemma}
\label{prop:cbound}
Let $f(z) = a(z^3/3 - Dz) + c \in \QQ[z]$ with $\pm aD \in \left\{ \frac 3 4, \frac 3 2, \frac 9 4, 3, \frac{15}{4}, \frac 9 2, \frac{21}{4} \right\}$. If $f$ is  postcritically finite, then
$|c|^2 < 11 |D|$.
\end{lemma}
\begin{proof}
$\Crit (f) = \lbrace \pm\sqrt{D}\rbrace$ and 
\[
f\left(\pm \sqrt{D} \right) = \mp \frac{2}{3} aD^{3/2} +c.
\] 
 A calculation shows that if $\abs{c}^2\geq 11\abs{D}$ and $\alpha \in \CC$ with $\abs{\alpha}>\abs{c}$, then $\abs{f(\alpha)}>\abs{\alpha}$. 
Since $a\neq 0$ then  at least one of the critical points $\gamma$ must satisfy $\abs{f(\gamma)}>\abs{c}$.
\end{proof}

\begin{lemma}
\label{prop:acvaluations}
Let $f(z) = a(z^3/3 - Dz) + c \in \QQ[z]$. If $f$ is $p$-adically post-critically bounded, then
\[
\abs{c\sqrt{a}}_p \leq 
\begin{cases}
1 & \text{ if } p \geq 5\\
3^{-1/2}& \text{ if } p = 3\\
2^{3} & \text{ if } p = 2.
\end{cases}
\]
\end{lemma}

\begin{proof}
Let $g= f^\phi$ for some $\phi \in \PGL_2(\bar \QQ)$ so that $f^\phi$ is monic and has a fixed point at 0. Then $g$ is of the form
\begin{equation}
\label{eqn:monicform}
g(z) = z^3 +3\alpha z^2 +(3\alpha^2-aD)z
\end{equation}
 where $ \alpha$  is a root of the polynomial
\begin{equation}
\label{eqn:alphapoly}
z^3 -(aD+1) z + c \sqrt{\frac a 3}.
\end{equation}
The critical points for $g$ are now $-\alpha \pm \sqrt{\frac{aD}3}$.

From~\cite[Theorems 1.2 and 4.1]{MR3040666}, we know that if $g$
 is $p$-adically post-critically bounded, the critical points must satisfy
\[
\abs{-\alpha \pm \sqrt{\frac{aD}3} }_p \leq
\begin{cases}
1 & \text{ if } p >2 \\
2 & \text{ if } p = 2.
\end{cases}
\]

First consider $p \geq 3$. Add the critical points to see that
\[
\abs{-2\alpha}_p \leq 1, \quad \text{ so } \abs{\alpha  }_p \leq 1.
\]

Therefore, the polynomial in Equation~\eqref{eqn:alphapoly} is monic and all three roots lie in the $p$-adic unit disk. A Newton polygon argument says that the coefficients of that polynomial must also lie in the $p$-adic unit disk: if any coefficient had negative valuation, some segment of the Newton polygon would have positive slope, which would imply that the polynomial has a root of absolute value greater than one. 

Since the constant term lies in the $p$-adic unit disk, 
 \[
 \abs{ c\sqrt{\frac a 3} }_p \leq 1.
 \]
 That gives the following bounds for $p\neq 2$:
\[
\abs{c\sqrt{a}}_p \leq 
\begin{cases}
1 & \text{ if } p \geq 5\\
3^{-1/2}& \text{ if } p = 3.\\
\end{cases}
\]

Now consider the case $p=2$. We have
\begin{equation}\label{eqn:2adic_bound}
\abs{ -\alpha \pm \sqrt{\frac{aD}3} }_2 \leq2.
\end{equation}

Using the list of possible $aD$ values from Proposition~\ref{prop:aD_vals}, we see that 
\[
\abs{ \sqrt{\frac{aD}3} }_2 \leq 2.
\]
 Applying the ultrametric triangle inequality to Equation~\eqref{eqn:2adic_bound}  yields $\abs{ -\alpha}_2 \leq 2$. Therefore the Newton polygon for that polynomial 
 at $p=2$ can have a segment of slope at most 1. Since the polynomial in Equation ~\eqref{eqn:alphapoly} is cubic, that means  the constant term must satisfy
\[
v_2\left( c\sqrt{\frac a 3} \right) \geq -3. 
\]
So $\abs{c\sqrt{a}}_2 \leq 2^{3}$, which completes the proof.
\end{proof}

  \section{The algorithms}  \label{cubicpcfsection}
This section presents algorithms for finding all bicritical cubic PCF polynomials over $\QQ[z]$; the algorithms depend on normal forms found in Section~\ref{sec:normal} and coefficient bounds proven in Section~\ref{sec:CoeffBounds}.

\subsection{Case 1: Rational critical points}
Results in Section~\ref{sec:CoeffBounds} give a finite set of coefficients to test, so the first algorithm is straightforward.
\begin{theorem}
\label{cubicbicriticalthm-rat}If $f(z) \in \QQ [z]$ is a cubic bicritical PCF polynomial with rational critical points, then $f(z)$ is conjugate to 
$f_{a,c}(z) = a(-2z^3+3z^2)+c$ where 
\begin{small}
\[
(a,c) \in \left\lbrace (1,0), \left(\pm1,\frac 12 \right),  \left(\frac 12, \pm 1\right), \left(2, -\frac 12\right), \left(\frac 32, 0\right), (-1,1), \left(-2,\frac 32\right), \left(-\frac 32, 1\right), \left(-\frac 12, 0\right)\right\rbrace.
\]
\end{small}

\end{theorem}
\begin{proof}
From Proposition ~\ref{familyprop}, we know that every cubic polynomial in $\QQ[z]$ with rational critical points is conjugate to a map of the form $f_{a,c}(z) = a(-2z^3+3z^2)+c$ for some $a,c \in \QQ$. 

From Proposition~\ref{prop:Belyiformalist},  if $f_{a,c}$ is post critically bounded in every place, then
$$\pm a\in \left\lbrace  \frac12,  1,  \frac 32,  2,  \frac 5 2, 3,  \frac{7}{2}\right\rbrace \text{ and }
\pm c \in \left\lbrace 0 ,  1,  \frac 12,  \frac 32,  2 \right\rbrace.$$

This gives $126$ possibilities for $(a,c)$.
 The authors used built-in Sage~\cite{sage} functionality\footnote{Sage code is available with the arXiv distribution of this article.} to test all such pairs.
\end{proof}

\subsection{Case 3: Irrational critical points}
This case is more delicate. Results in Section~\ref{sec:CoeffBounds} give  relationships between absolute values of the coefficients for cubic PCF maps. We must disentangle these relationships to build a finite search space.

\begin{theorem}
\label{cubicbicriticalthm-irrat}
If $f(z) \in \QQ [z]$ is a cubic bicritical PCF polynomial that is not conjugate to a polynomial with rational critical points, then $f(z)$ is conjugate to 
\begin{small}
$f_{D,a,c}(z) = a(\frac{z^3}{3}-Dz)+c$ where $$(D,a,c) \in \left\lbrace \left(2, -\frac34, 2\right), \left(-7, -\frac{3}{28}, \frac 72\right) \right\rbrace.$$
\end{small}
\end{theorem}

\begin{proof} 
From Proposition~\ref{prop:MoveCritPts}, we know that every cubic polynomial in $\QQ[z]$ with irrational critical points is conjugate to a map of the form $f_{D,a,c}(z) = a(\frac{z^3}{3}-Dz)+c$ for some $a,c \in \QQ$ and a squarefree integer $D$. 

Note that if $c=0$, then $f_{D,a,c}(z)$ is conjugate to a cubic polynomial with rational critical points via conjugation by $\phi (z) = \frac{a-\sqrt{D}}{-z\sqrt{D}}$. Furthermore, $f_{D,a,-c}(z)$ is conjugate to $f_{D,a,c}(z)$, so we may assume that $c>0$. Therefore, we build a list of triples  $(D,a,c)$ with $D,a,c >0$, and each triple corresponds to four possibly PCF polynomials (varying the signs of $D$ and $a$).  
We split the algorithm into two cases corresponding to $D$ even and $D$ odd. 

\begin{description}
\item[Step 1] {\bf Loop over possible $aD$ values.}
From Proposition~\ref{prop:aD_vals}, if $f_{D,a,c}$ is PCF then:
\[
\pm aD \in \left\{ \frac 3 4, \frac 3 2, \frac 9 4, 3, \frac{15}{4}, \frac 9 2, \frac{21}{4} \right\}.
\]

\item[Step 2]\label{step:a2val} {\bf Compute $|a|_2$.}
 We use the value of $aD$ in  Step 1 and the parity of $D$.

\item [Step 3] {\bf Find an upper bound for  $|c|_p$ for each prime $p$.}
From  Lemma~\ref{prop:acvaluations}, we know that if $f_{D,a,c}(z)$ is $p$-adically post-critically bounded, then
\begin{equation}\label{eqn:adic-cond}
\abs{c\sqrt{a}}_p \leq 
\begin{cases}
1 & \text{ if } p \geq 5\\
3^{-1/2}& \text{ if } p = 3\\
2^{3} & \text{ if } p = 2.
\end{cases}
\end{equation}
So from  Step 2 we can find $e \leq 3$ such that $|c|_2 \leq 2^e$. Also using the list in Step 1, we conclude that $\abs{c}_p \leq 1$ for each prime $p\geq 3$.

\item [Step 4] {\bf Factor $D$ and $c$.}
Write $D=mP$ or $D=2mP$, where $m$ and $P$ are relatively prime odd squarefree integers such that $m$ divides the numerator of $aD$ and $P$ divides the denominator of $a$.  By equation~\eqref{eqn:adic-cond},  $P$ must also divide the numerator of $c$. Thus,  $c=\frac{Pk}{2^e}$ for some positive integer $k$. Note: For a fixed  $aD$ from the list above, $m$ comes from a finite set, but for now $P$ and $k$ can be arbitrarily large.

\item [Step 5] {\bf Bound the factors of $D$ and $c$.}
  From Proposition~\ref{prop:cbound}, we know that if $f_{D,a,c}(z)$ is post-critically bounded at the archimedean place, then  $|c|^2 < 11 |D|$.
Depending on the parity of $D$, this gives 
 \[
 \frac{P^2k^2}{2^{2e}} < 11mP \quad \text{ or }\quad  \frac{P^2k^2}{2^{2e}} < 22mP.
 \]
So $ Pk^2 < B$ where 
 $B = 11m\cdot 2^{2e} $  when $D$ is odd, and  $ B=11m \cdot 2^{2e+1}$  when $D$ is even.
  We know $e$   from Step 3, so for each $m$ we have an explicit value for the upper bound $B$.

\item [Step 6]  {\bf Loop over  $P$ values.}
For all odd, squarefree integers  $P < B$, we determine the set of possible $k$ values such that $Pk^2<B$.

\item [Step 7] {\bf Create the triple.} Each triple $(m, P, k)$ yields a triple $(D,a,c) = (mP, \frac{aD}{mP}, \frac{Pk}{2^e})$ or $(D,a,c) = (2mP, \frac{aD}{2mP}, \frac{Pk}{2^e})$. Finally, check that $3|ac$ to verify that the triple satisfies the $3$-adic condition in equation~\eqref{eqn:adic-cond}. If so, add $(D,a,c)$ to the list of possible PCF triples.
\end{description}

This algorithm yields a list of 5,957 triples corresponding to 23,828 possibly PCF polynomials. The authors used built-in Sage~\cite{sage} functionality to test all such triples.
 Only the two listed in the theorem statement above are actually PCF and are not conjugate to a polynomial already found in Theorem~\ref{cubicbicriticalthm-rat}.
\end{proof}

 Combining the results in Theorems~\ref{uniprop} , \ref{cubicbicriticalthm-rat}, and~\ref{cubicbicriticalthm-irrat} yields a total of 15 conjugacy classes of PCF cubic polynomials over $\QQ[z]$, and this list is exhaustive. In the table below, we provide one representative of each conjugacy class along with the critical portrait for the polynomial. In the portrait, the critical points are given by $\gamma$ and other points in the post-critical set are denoted $\bullet$. The monic centered form is given when it is defined over $\QQ[z]$; these appeared in~\cite{ingram2010}.

\begin{longtable}[b]{m{.2\paperwidth} m{.25\paperwidth}  m{.2\paperwidth} }
\hline
PCF polynomial   & Critical portrait  & Monic centered form\\
\hlineƒ
  $z^3$  &  \begin{tikzpicture}
	   \node (0) at (0,0) {\tiny$\gamma_1$};
            \draw [black, ->] (0) to [out = 0, in = 90, distance = .75 cm] (0);
            \end{tikzpicture}
            & $z^3$\\
\hline
$-z^3+1$ &  
            \begin{tikzpicture}
            \node (blah) at (0,1/2)  { };
            \node (0) at (0,0) {\tiny$\gamma_1$};
            \node (1) at (1,0) {\tiny$\bullet$};

            \draw [black, ->] (0) to [out = -45, in = 225] (1);
            \draw [black, ->]   (1) to[out=135,in=45] (0);
            \end{tikzpicture} &\\
\hline
$-2z^3+3z^2+\frac12$  &  
            \begin{tikzpicture}
 \node (0) at (0,0) {\tiny$\gamma_1$};
        \node (1) at (1.5,.5) {\tiny$\gamma_2$};
        \node (12) at (.75,0) {\tiny$\bullet$};
        \node (32) at (1.5,-.5) {\tiny$\bullet$};

        \draw [black, ->] (0) to (12);
        \draw [black, ->]   (12) to[out=90,in=180] (1);
        \draw [black, ->]   (1) to[out=0,in=0] (32);
        \draw[black, ->] (32) to [out = 180, in = 270] (12);
            \end{tikzpicture} & \\ 
  \midrule
 $-z^3+\frac{3}{2}z^2+1$ &   \begin{tikzpicture}
            \node (0) at (0,0) {\tiny$\gamma_1$};
            \node (1) at (1,0) {\tiny$\gamma_2$};
            \node (32) at (2,0) {\tiny$\bullet$};

            \draw [black, ->] (0) to (1);
            \draw [black, ->]   (1) to[out=-90,in=-90] (32);
            \draw[black, ->] (32) to [out = 90, in = 90] (1);
        \end{tikzpicture}  &\\
  \midrule
 $-2z^3+3z^2$   &    \begin{tikzpicture}
            \node (0) at (0,0) {\tiny$\gamma_1$};
            \node (1) at (2,0) {\tiny$\gamma_2$};
            \draw [black, ->] (0) to [out =0, in = 90, distance = .75 cm] (0);
            \draw [black, ->]   (1) to[out=0,in=90, distance = .75 cm] (1);
        \end{tikzpicture} 
         & $z^3+\frac 32 z$ \\
  \midrule        
$-3z^3+\frac 92 z^2$  &  
        \begin{tikzpicture}
            \node (0) at (2,0) {\tiny $\gamma_1$};
            \node (1) at (0,0) {\tiny$\gamma_2$};
            \node (32) at (1,0) {\tiny$\bullet$};

            \draw [black, ->] (1) to (32);
            \draw [black, ->]   (32) to (0);
            \draw[black, ->] (0) to [out = 0, in = 90, distance = .75 cm ] (0);
        \end{tikzpicture} &\\
   \midrule
$-z^3+\frac 32z^2-1$ & 
  \begin{tikzpicture}
            \node (0) at (0,0) {\tiny $\gamma_1$};
            \node (-1) at (.75,0) {\tiny$\bullet$};
            \node (32) at (1.75,0) {\tiny$\bullet$};
            \node (1) at (2.5,0) {\tiny$\gamma_2$};
            \node (12) at (3.25,0) {\tiny$\bullet$};
            \draw [black, ->] (0) to (-1);
            \draw [black, ->]   (-1) to[out=-90,in=-90] (32);
            \draw[black, ->] (32) to [out = 90, in = 90] (-1);
            \draw[black, ->] (1) to (12);
            \draw[black, ->] (12) to [out = 0, in = 90, distance = .75 cm] (12);
        \end{tikzpicture} &\\
  \midrule
$-4z^3+6z^2-\frac 12$  &
\begin{tikzpicture}
            \node (0) at (0,0) {\tiny $\gamma_1$};
            \node (12) at (.75,0) {\tiny$\bullet$};
            \node (32) at (1.75,0) {\tiny$\bullet$};
            \node (1) at (2.5,0) {\tiny$\gamma_2$};
            \draw [black, ->] (0) to (12);
            \draw[black, ->] (12) to [out= -45, in =225] (32);
            \draw [black, ->]   (32) to[out=135,in=45] (12);
            \draw[black, ->] (1) to  (32);
        \end{tikzpicture}
        & $z^3+3z$\\
  \midrule
$2z^3-3z^2+1$ & 
        \begin{tikzpicture}
            \node (0) at (0,0) {\tiny $\gamma_1$};
            \node (1) at (1,0) {\tiny$\gamma_2$};

            \draw [black, ->] (0) to [out = -45, in = 225] (1);
            \draw [black, ->]   (1) to[out=135,in=45] (0);
        \end{tikzpicture}
         & $z^3-\frac 32 z$\\
  \midrule
$4z^3-6z^2+\frac 32$  &    \begin{tikzpicture}
            \node (0) at (0,0) {\tiny $\gamma_1$};
            \node (32) at (1,0) {\tiny$\bullet$};
            \node (1) at (2,0) {\tiny$\gamma_2$};
            \node (-12) at (3,0) {\tiny$\bullet$};
            \draw [black, ->] (0) to (32);
            \draw [black, ->]   (32) to[out=0,in=90, distance = .75cm ] (32);
            \draw[black, ->] (1) to (-12);
            \draw[black, ->] (-12) to [out = 0, in = 90, distance = .75 cm] (-12);
        \end{tikzpicture}
        & $z^3-3z$ \\
  \midrule
$2z^3-3z^2+\frac 12$ &  \begin{tikzpicture}
            \node (0) at (0,0) {\tiny $\gamma_1$};
            \node (12) at (1,0) {\tiny$\bullet$};
            \node (1) at (2,0) {\tiny$\gamma_2$};
            \node (-12) at (3,0) {\tiny$\bullet$};
            \draw [black, ->] (0) to [out = -45, in = -135] (12);
            \draw [black, ->]   (12) to[out=135,in=45, ] (0);
            \draw[black, ->] (1) to (-12);
            \draw[black, ->] (-12) to [out = 0, in = 90, distance = .75 cm] (-12);
        \end{tikzpicture} &\\
  \midrule
 $3z^3-\frac 92z^2+1$ & 
        \begin{tikzpicture}
            \node (0) at (0,0) {\tiny $\gamma_1$};
            \node (1) at (1,0) {\tiny$\gamma_2$};
            \node (-12) at (2,0) {\tiny$\bullet$};
            \draw [black, ->] (0) to (1);
            \draw [black, ->]   (1) to (-12);
            \draw[black, ->] (-12) to [out = 0, in = 90, distance = .75 cm] (-12);
        \end{tikzpicture} &\\
  \midrule
  $z^3-\frac 32z^2$&      \begin{tikzpicture}
            \node (0) at (0,0) {\tiny $\gamma_1$};
            \node (1) at (2,0) {\tiny $\gamma_2$};
            \node (-12) at (3,0) {\tiny$\bullet$};
            \draw [black, ->] (0) to [out = 0, in = 90, distance = .75 cm ](0);
            \draw[black, ->] (1) to (-12);
            \draw[black, ->] (-12) to [out = 0, in = 90, distance = .75 cm] (-12);
        \end{tikzpicture} 
         & $z^3-\frac 34 z + \frac 34$ and $z^3-\frac 34 z - \frac 34$\\
  \midrule
$-\frac{1}{4}z^3+\frac32 z+2$ &  \begin{tikzpicture}
            \node (rt2) at (0,0) {\tiny $\gamma_1$};
            \node (2plus) at (.75,0) {\tiny$\bullet$};
            \node (-2rt) at (1.5,0) {\tiny$\bullet$};
            \draw [black, ->] (rt2) to (2plus);
            \draw[black, ->] (2plus) to [out= -45, in =225] (-2rt);
            \draw [black, ->]   (-2rt) to[out=135,in=45] (2plus);
            \node (-rt2) at (2.25,0) {\tiny $\gamma_2$};
            \node (2minus) at (3,0) {\tiny$\bullet$};
            \node (2rt) at (3.75,0) {\tiny$\bullet$};
            \draw [black, ->] (-rt2) to (2minus);
            \draw[black, ->] (2minus) to [out= -45, in =225] (2rt);
            \draw [black, ->]   (2rt) to[out=135,in=45] (2minus);
        \end{tikzpicture}&\\
  \midrule
$-\frac{1}{28}z^3-\frac34 z+\frac 72$& 
 \begin{tikzpicture}
            \node (sqrt-7) at (0,0) {\tiny $\gamma_1$};
            \node (7minus) at (1,0) {\tiny$\bullet$};
            \node (-sqrt-7) at (2,0) {\tiny $\gamma_2$};
            \node (7plus) at (3,0) {\tiny$\bullet$};
            \draw [black, ->] (sqrt-7) to [out = -45, in = -135] (7minus);
            \draw [black, ->]   (7minus) to[out=135,in=45, ] (sqrt-7);
            \draw [black, ->] (-sqrt-7) to [out = -45, in = -135] (7plus);
            \draw [black, ->]   (7plus) to[out=135,in=45, ] (-sqrt-7);
        \end{tikzpicture} & \\
  \midrule
\caption{Critical Portraits of Cubic PCF Polynomials over $\QQ$}\label{CubicTable}

\end{longtable}

\bibliographystyle{plain}
\bibliography{CubicPCFFinal}

\end{document}